\documentclass[12pt,reqno]{amsart}
\usepackage{amssymb}
\usepackage[curve,matrix,arrow]{xy}
\newtheorem{thm}{Theorem}[section]
\newtheorem{lemma}[thm]{Lemma}
\newtheorem{cor}[thm]{Corollary}
\newtheorem{prop}[thm]{Proposition}
\newtheorem{conjj}[thm]{Conjecture}
\newtheorem{defi}[thm]{Definition}

\theoremstyle{definition}
\newtheorem{rem}[thm]{Remark}

\newtheorem{exm}[thm]{Example}

\def\bZ{{\mathbb Z}}

\def\ovl#1{\overline{#1}}

\def\res{\operatorname{Res}\nolimits}

\def\Hom{\operatorname{Hom}\nolimits}
\def\End{\operatorname{End}\nolimits}

\def\Ker{\operatorname{Ker}\nolimits}

\def\proj{\operatorname{(proj)}\nolimits}

\def\stmod{\operatorname{stmod}\nolimits}
\def\SL{\operatorname{SL}\nolimits}

\def\PSL{\operatorname{PSL}\nolimits}

\def\SP#1#2{\operatorname{Sp}\nolimits}

\def\conj{\operatorname{conj}\nolimits}
\def\stmod{\operatorname{stmod}\nolimits}

\def\bZ{{\mathbb Z}}
\def\bfq{{\mathbb F}_q}

\def\dmd{\diamondsuit}

\newcommand{\ls}[2]{{}^{{#1}}\!{{#2}}}

\begin{document}

\title[The torsion group of endotrivial modules]
{The torsion group of endotrivial modules} 

\author{Jon F. Carlson$^*$}\thanks{$^*$~Partially supported by a grant from NSF}
\author{Jacques Th\'evenaz}
\date\today

\subjclass[2000]{Primary 20C20}

\begin{abstract}
Let $G$ be a finite group and let $T(G)$ be the abelian group of 
equivalence classes of endotrivial $kG$-modules, where $k$ is an 
algebraically closed field of characteristic~$p$. We determine, in
terms of the structure of $G$, 
the kernel of the restriction map from $T(G)$ to $T(S)$, 
where $S$ is a Sylow
$p$-subgroup of~$G$, in the case when $S$ is abelian. 
This provides a classification of all
torsion endotrivial $kG$-modules in that case.
\end{abstract}

\maketitle


\section{Introduction} \label{sec:intro} 

\noindent
Endotrivial modules for a finite group~$G$ over a field $k$ of 
prime characteristic~$p$ play a significant role
in modular representation theory. Among other things, they form an
important part of the Picard group of self-equivalences
of the stable category $\stmod(kG)$ of finitely generated $kG$-modules
modulo projectives. They are modules which have universal deformation
rings \cite{SL}. The endotrivial modules  have been 
classified in case $G$ is a $p$-group (\cite{CT1, CT2})
and various results have appeared since 
for some specific families of groups
\cite{CHM, CMN1, CMN2, CMN3, CMT1, CMT2, CMT3, Ma, MT, NR, LM2}.
Recently, another line of research has developed that is concerned with the 
classification of all endotrivial modules which are simple
\cite{Ro, LMS, LM1}.

The main purpose of this paper is to classify all torsion 
endotrivial modules when a Sylow $p$-subgroup is abelian.
We let $G$ be a finite group and $T(G)$ be the abelian group 
of  equivalence classes of endotrivial $kG$-modules,
where $k$ is an algebraically closed field of characteristic~$p$.
Let $S$ be a Sylow $p$-subgroup of~$G$. We fix the notation
\[
K(G)=\Ker \big\{\res_S^G : T(G) \longrightarrow T(S) \, \big\} \,.
\]
One of the main open questions is to describe this kernel 
explicitly and we achieve this goal here in the case that $S$ is abelian.
Actually, $K(G)$ is known to be equal to 
the torsion subgroup of $T(G)$ in most cases.
Specifically, this happens whenever $S$ is not cyclic, 
generalized quaternion, or semi-dihedral
(because, if we exclude these three cases, 
then $T(S)$ is torsion-free by~\cite{CT2}).
The excluded cases are treated in~\cite{MT}, \cite{Kaw}  and \cite{CMT2}.

Let $N = N_G(S)$ denote the normalizer of a Sylow $p$-subgroup $S$ of~$G$.
It is known that the restriction map 
$\res_{N}^G : K(G) \longrightarrow K(N)$ is injective,
induced by the Green correspondence, 
and the problem is to describe its image.
In fact, $K(N)$ consists of all 
one-dimensional representations of~$N$
and the main difficulty is to know which of 
them lie in the image of $\res_{N}^G$.
That is, which of them have Green 
correspondents that are endotrivial.
Indeed, if $J$ is the intersection of the kernels of the 
one-dimensional $kN$-modules whose Green correspondents
are endotrivial, then $K(G)$ is isomorphic to the dual
group $(N/J)^*$. So the problem of finding 
$K(G)$ comes down to the question of what is~$J$.

Another approach was introduced by Balmer in \cite{Ba1},
in which he shows that $K(G)$ is isomorphic
to the group $A(G)$ of all weak homomorphisms $G\to k^\times$
(defined below in Section~\ref{sec:weak}). Balmer's method was used
effectively in \cite{CMN3} to compute $K(G)$ in some crucial
cases for $G$ a special linear group. The method involved
the construction of a system of local subgroups $\{\rho^i(Q)\}$
indexed on the collection of nontrivial subgroups $Q$ of the 
Sylow subgroup $S$ of $G$ and for $i \geq 1$
(see Section~\ref{sec:local}, below for the definition). 
These subgroups are in the kernels of
all weak homomorphisms and we show here that $\rho^i(S) \subseteq
N_G(S)$ is in the kernel of all one-dimensional representations of~$N_G(S)$
whose Green correspondents are endotrivial modules. That is, 
$\rho^i(S) \subseteq J$, for $J$ defined as above. 

Hence, the question of determining $K(G)$ becomes: is $J$ equal to 
$\rho^\infty(S)$, the limit of the system? The answer is yes in all
examples that we know. The main theorem of this paper says that 
the answer is yes whenever the Sylow $p$-subgroup of $G$ is abelian. 
Indeed, we prove more. We show that $J = \rho^2(S)$ the subgroup of 
$N = N_G(S)$ generated by the commutator subgroup $[N,N]$, by~$S$
and by all intersections $N \cap [N_G(Q), N_G(Q)]$ for $Q$ a nontrivial
subgroup of~$S$. Thus, $K(G) \cong (N/\rho^2(S))^*$ in the case that 
$S$ is abelian. Balmer's characterization of $K(G)$ in terms of weak
homomorphism is crucial for the proof, which appears in Section~\ref{sec:mainthm}.

In Section~\ref{sec:cyclic} we show that the main theorem can be used
to describe $K(G)$ explicitly when $S$ is cyclic. 
An extension of the main theorem
to the case where the normalizer $N$ of $S$ controls $p$-fusion is 
given in Section~\ref{sec:fusion}. The paper ends with some examples
of simple or almost simple groups where the subgroup $K(G)$ is not 
trivial.

\smallskip
\noindent
{\em Acknowledgment}~: The authors are grateful 
to Paul Balmer and Caroline Lassueur for several useful discussions.


\section{Endotrivial modules and the restriction to the 
Sylow subgroup} \label{sec:prelim}

\noindent
Throughout this paper, $k$ denotes an algebraically 
closed field of prime 
characteristic~$p$ and $G$ is a finite group.
We assume that all modules are finitely generated.
If $M$ and $L$ are $kG$-modules, the notation 
$M\cong L \oplus \proj$ means
that $M$ is isomorphic to the direct sum of 
$L$ with some projective $kG$-module, which might be zero.
We write $k$ for the trivial $kG$-module. 
Unless otherwise specified, the symbol $\otimes$
is the tensor product $\otimes_k$ of the underlying vector spaces.
In  the case that $M$ and $L$ are $kG$-modules, the tensor product 
$M \otimes L$ is a $kG$-module with $G$ acting by the diagonal 
action on the factors.

We assume that $G$ has order divisible by~$p$ 
and we let $S$ be a Sylow $p$-subgroup of~$G$.
Recall that a $kG$-module $M$ is endotrivial provided 
its endomorphism algebra $\End_k(M)$ is isomorphic (as a $kG$-module)
to the direct sum of the trivial module~$k$ and a projective $kG$-module.
In other words, $M$ is endotrivial if and 
only if $\Hom_k(M,M) \cong M^* \otimes M \cong k \oplus \proj$,
where $M^*$ denotes the $k$-dual of~$M$.
Any endotrivial module $M$ splits as the direct sum $M=M_0 \oplus \proj$
for an indecomposable endotrivial $kG$-module 
$M_0$, which is unique up to isomorphism.
We let $T(G)$ be the set of equivalence classes 
of endotrivial $kG$-modules for the equivalence relation
$$M \sim L \Longleftrightarrow M_0 \cong L_0 \,.$$
Every equivalence class contains a unique 
indecomposable module up to isomorphism.
The tensor product induces an abelian group structure on the set $T(G)$,
written additively as $[M]+[L] = [M \otimes L]$.
The zero element of $T(G)$ is the class $[k]$ of the trivial module,
while the inverse of the class of a module 
$M$ is the class of the dual module~$M^*$.
The group $T(G)$ is known to be a finitely generated abelian group.
The following result is Lemma 2.3 in~\cite{CMT1}.

\begin{lemma} \label{lem:K(G)}
Let $K(G)$ be the kernel of the restriction map 
$\res_S^G : T(G) \longrightarrow T(S)$ to a Sylow $p$-subgroup~$S$.
\begin{itemize}
\item[(a)] $K(G)$ is a finite subgroup of~$T(G)$.
\item[(b)] $K(G)$ is the entire torsion subgroup $TT(G)$ of~$T(G)$,
provided $S$ is not cyclic, generalized quaternion, or semi-dihedral.
\end{itemize}
\end{lemma}

We say that a $kG$-module $M$ has trivial Sylow restriction, 
if the restriction of $M$ to a Sylow $p$-subgroup $S$
has the form $M{\downarrow_S^G} \cong k \oplus \proj$. 
Any such module is endotrivial and is the direct sum
of an indecomposable trivial source module and a projective module. 
Thus $M$ has trivial Sylow restriction
if and only if its class $[M]$ belongs to~$K(G)$.

\begin{prop} \label{prop:normal1}
Suppose that a finite group $H$ has a nontrivial normal $p$-subgroup.
Then every indecomposable $kH$-module with 
trivial Sylow restriction has dimension one.
\end{prop}

\begin{proof}
The proof is a straightforward application of the 
Mackey formula. The details appear in Lemma 2.6 in \cite{MT}.
\end{proof}

In the situation of Proposition~\ref{prop:normal1}, 
for any one-dimensional $kH$-module $L$,
we write $\chi_L:H\to k^\times$ for 
the corresponding group homomorphism (representation).

Our next result is an easy application of the 
Green correspondence. For details, see Proposition 2.6 in~\cite{CMN1}.

\begin{prop} \label{prop:Green}
Let $S$ be a Sylow $p$-subgroup of~$G$ and let $N = N_G(S)$.
\begin{itemize}
\item[(a)] The restriction map 
$\res_{N}^G : T(G) \to T(N)$ is injective, 
induced by the Green correspondence.
\item[(b)] In particular, the restriction map 
$\res_{N}^G : K(G) \to K(N)$ is injective.
\end{itemize}
\end{prop}

We emphasize, that if $M$ and $L$ are $kG$-modules with trivial 
Sylow restriction, by Proposition~\ref{prop:normal1}, 
$M{\downarrow_{N}^G} \cong U \oplus \proj$ 
and $L{\downarrow_{N}^G} \cong V \oplus \proj$, where $U$ and $V$
are $kN$-module of dimension one. Here $U$ is the Green correspondent
of $M$, $V$ is the Green correspondent of $L$, and we see 
automatically that $U\otimes V$ is the Green correspondent of 
$(M \otimes L)_0$, the unique indecomposable nonprojective direct
summand of $M \otimes L$. 

We know that $K(N_G(S))$ consists exactly
of all one-dimensional representations of~$N$.
The main problem is to know which of them are in 
the image of the restriction map from~$K(G)$.
In other words, given a one-dimensional $kN$-module~$U$,
we need to know when its Green correspondent~$M$ is endotrivial.

Another way of viewing the situation is the following.

\begin{prop} \label{prop:charsub}
Let $S$ denote a Sylow $p$-subgroup of $G$ and let 
$N = N_G(S)$. Let $J \subseteq N$ be the intersection of 
the kernels of all one-dimensional $kN$-modules $U$ such
that the Green correspondent $M$ of $U$ is an endotrivial
$kG$-module. That is, $J$ is the intersection of the kernels of 
$U$ such that $[U]$ is in the image of the restriction
$\res_{N}^G : T(G) \to T(N)$. Then $K(G) \cong (N/J)^*
\cong \Hom(N/J, k^\times)$. 
\end{prop}

\begin{proof}
The proof is straightforward. The restriction map 
$\res_{N}^G : K(G) \to K(N)$, being injective, 
gives an isomorphism between $K(G)$ and a subgroup
$A \subseteq K(N) \cong \Hom(N,k^\times)$, the group of one-dimensional 
representations of~$N$. But $A$ is isomorphic to the dual 
group of $N/J$ where $J$ is the intersection of the 
kernels of the elements of $A$.
\end{proof}

{}From the proposition, we see that the problem of 
characterizing the group $K(G)$ is equivalent to finding
the group $J$. The main purpose of this paper is to
offer a possible candidate for the subgroup $J$. In 
addition, we prove that the candidate is, in fact, equal to $J$
in the case that the Sylow $p$-subgroup $S$ is abelian.


\section{Weak homomorphisms and the kernel of restriction}\label{sec:weak}
In \cite{Ba1}, Balmer provided a new characterization of the 
group $K(G)$ in terms of the group of weak $S$-homomorphisms, which we
describe in this section. Note that 
Balmer's construction is more general than
the one that we use here. He defined ``weak $H$-homomorphisms'', for
any subgroup $H$ containing $S$. Because, we only deal with the 
case that $H=S$ in this paper, we call them ``weak homomorphisms''. 
Note that Balmer has expanded his results in~\cite{Ba2}, 
but here we only need the formulation in~\cite{Ba1}.

For notation, recall that $\ls gS$ denotes 
the conjugate subgroup $gSg^{-1}$,
while  $S^g$ denotes $g^{-1}Sg$. 

\begin{defi} \label{def:weakhomo}
A map $\chi:G\to k^\times$ is called a {\em weak homomorphism} 
(``weak $S$-homomorphism'' in the language of \cite{Ba1})
if it satisfies the following three conditions:
\begin{itemize}
\item[(a)] If $s\in S$, then $\chi(s)=1$.
\item[(b)] If $g\in G$ and $S\cap \ls gS=\{1\}$, then $\chi(g)=1$.
\item[(c)] If $a,b\in G$ and if $S\cap \ls aS \cap \ls {ab}S\neq \{1\}$, 
then $\chi(ab)=\chi(a)\chi(b)$.
\end{itemize}
\end{defi}

The product of two weak homomorphisms $\chi$ and $\psi$ is 
defined by $(\chi\psi)(g)=\chi(g)\psi(g)$
and is again a weak homomorphism.  The set 
$A(G)$ of all weak homomorphisms is an abelian group
under this operation. 

\begin{thm} \label{thm:Balmer} (Balmer~\cite{Ba1})
The groups $K(G)$ and $A(G)$ are isomorphic.
\end{thm}

The isomorphism is explicit and is described in details in~\cite{Ba1}.
In particular, it is shown that given a weak homomorphism $\chi$, there 
is a certain norm-type formula using $\chi$ that constructs a 
homomorphism from the permutation module $k(G/S)$ to itself, whose 
image is the endotrivial module associated to~$\chi$.

It is important, in what follows, to understand 
Balmer's isomorphism on restriction to a subgroup~$H$
having a nontrivial normal $p$-subgroup. This is our next result.
For notation, given a finite group $H$, we let 
$\dmd(H)$ be the smallest normal subgroup of~$H$
such that 
$H/\dmd(H)$ is an abelian $p'$-group.
In other words 
$\dmd(H) = [H,H]S$ is the subgroup of~$H$ generated 
by the commutator subgroup $[H,H]$
and by a Sylow $p$-subgroup $S$ of~$H$.

\begin{prop} \label{prop:normal}
Suppose that a finite group $H$ has a nontrivial normal $p$-subgroup.
\begin{itemize}
\item[(a)] Every weak homomorphism 
$\chi: H\to k^\times$ is a group homomorphism.
\item[(b)] The isomorphism $K(H)\to A(H)$ maps 
the class of a one-dimensional $kH$-module~$M$
to the corresponding group representation $H\to k^\times$.
\item[(c)] The group $A(H)$ is isomorphic to the group 
of one-dimensional representations of~$H$, that is,
the dual group $(H/\dmd(H))^*$.
\end{itemize}
\end{prop}

\begin{proof}
To prove (a), let $Q$ be a nontrivial normal $p$-subgroup of~$H$.
For any $a,b\in G$, the subgroup 
$S\cap \ls aS \cap \ls {ab}S$ is nontrivial
because it contains~$Q=\ls a Q=\ls{ab}Q$.
Thus the third condition in Definition \ref{def:weakhomo}
implies that $\chi(ab)=\chi(a)\chi(b)$.

Statement (b) follows from the construction of the 
isomorphism given in Section~2.5 of~\cite{Ba1}.
See also Corollary~5.1 in~\cite{Ba1}.
The proof of (c) is straightforward, because the image 
of any group homomorphism $H\to k^\times$
is contained in the subgroup of~$k^\times$ consisting of 
roots of unity, namely the group of all $p'$-roots of unity
since $k$ is algebraically closed and has characteristic~$p$.
\end{proof}

\begin{prop} \label{prop:injective}
Let $N = N_G(S)$ where $S$ is the Sylow $p$-subgroup of $G$.
The restriction map 
$\res_{N}^G : A(G) \to A(N)$ is injective.
\end{prop}

\begin{proof}
This follows from Proposition~\ref{prop:Green}, 
Balmer's isomorphism in Theorem~\ref{thm:Balmer},
and the fact that this isomorphism is natural.
\end{proof}


\section{A system of local subgroups}  \label{sec:local}

\noindent
In this section, we discuss the properties of a sequence of  
subgroups $\rho^i(Q) \subseteq N_G(Q)$,
where $Q$ is a nontrivial $p$-subgroup of~$S$ and $i \geq 1$.
The construction of the sequence was first 
presented in~\cite{CMN3}, 
though the version here is slightly different. 
The definition of the subgroups $\rho^i(Q)$ 
depends not only on~$N_G(Q)$,
but involves all normalizers of nontrivial 
subgroups of the Sylow subgroup~$S$ of~$G$. 

The definition proceeds inductively as follows.
We fix a Sylow $p$-subgroup $S$ of~$G$.
For any nontrivial subgroup $Q$ of~$S$, let  
\[
\rho^1(Q) \; := \; 
\dmd(N_G(Q)) \,.
\]
As before, 
$\dmd(N_G(Q))$ is the product of the 
commutator subgroup of $N_G(Q)$
and a Sylow $p$-subgroup of~$N_G(Q)$. 
The original version in~\cite{CMN3} uses simply
the commutator subgroup of $N_G(Q)$. For $i >1$, we let
\[
\rho^i(Q) \; := \;\; <\,  
N_G(Q) \cap \rho^{i-1}(Q^\prime) \;\mid\; \{1\}\neq Q' 
\subseteq S  \,> \,,
\]
the subgroup generated by all the subgroups
$N_G(Q) \cap \rho^{i-1}(Q^\prime)$, 
for all nontrivial subgroups $Q^\prime$ of~$S$.
This contains $\rho^{i-1}(Q)$, so we 
have a nested sequence of subgroups
\[
\rho^1(Q) \subseteq \rho^2(Q) \subseteq \rho^3(Q) 
\subseteq \ldots \subseteq N_G(Q) \,.
\]
Since $G$ is finite, the sequence eventually 
stabilizes and we let $\rho^\infty(Q)$ be the limit subgroup of
the sequence $\{\rho^i(Q) \mid i\geq1\}$, namely their union.

The definition of the subgroups $\rho^i(Q)$, was originally
motivated in \cite{CMN3} by the following observation. 

\begin{prop} \label{prop:rho-weak}
Suppose that $\chi: G \to k^\times$ is a weak homomorphism 
as defined in the last section. If $x \in \rho^i(Q)$ for some
$i \geq 1$ and for some nontrivial subgroup $Q \in S$, then 
$\chi(x) = 1$. 
\end{prop}

\begin{proof}
In the case that $i = 1$, the statement is a trivial consequence
of the definition of a weak homomorphism. That is, $\chi(x) = 1$
for any $x \in Q$ by Definition \ref{def:weakhomo}(a), and is a 
homomorphism to an abelian group when restricted to $N_G(Q)$ by
\ref{def:weakhomo}(c).  So assume that $i > 1$ and that 
$\chi(x) = 1$ for all $x \in \rho^{i-1}(Q)$ for all nontrivial 
subgroups $Q \subseteq S$. Then $\chi(x) = 1$ for all $x \in 
N_G(Q) \cap \rho^{i-1}(Q^\prime)$ for any nontrivial subgroups 
$Q$ and $Q^\prime$. Thus $\chi(x) = 1$ for all $x \in \rho^i(Q)$,
and the proposition is proved by induction. 
\end{proof}

Suppose that $M$ is a $kG$-module with 
trivial Sylow restriction. By Proposition~\ref{prop:normal1},
for any nontrivial subgroup $Q$ of $S$, 
there is a one-dimensional $kN_G(Q)$-module $L_Q$ such that
$M{\downarrow_{N_G(Q)}^G} \cong L_Q \oplus\proj$. 
We write $\chi_Q=\chi_{L_Q}$ for the corresponding
group homomorphism $\chi_Q:N_G(Q)\to k^\times$.

The next result encapsulates the main idea of this section.
It should be compared with Proposition \ref{prop:rho-weak}.

\begin{prop} \label{prop:main}
Suppose that $M$ is a $kG$-module with trivial Sylow 
restriction, let $Q$ be a nontrivial subgroup of $S$
and let $\chi_Q:N_G(Q)\to k^\times$ be defined as above.
Then $\rho^\infty(Q)$ is contained in the kernel of~$\chi_Q$.
\end{prop}

\begin{proof}
We prove, by induction on~$i$, that if $x\in \rho^i(Q)$, then $\chi_Q(x)=1$.
In the case that $i=1$, the result is 
a consequence of the fact that
$\rho^1(Q) = 
\dmd(N_G(Q))$ is in the kernel of every 
one-dimensional character on $N_G(Q)$.
Inductively, we assume that the lemma is 
true for $\rho^j(Q')$, whenever $j < i$ and $1\neq Q'\subseteq S$.
Because $x  \in \rho^i(Q)$ we have 
that $x = x_1 \cdots x_m$ for some $m$,
where  $x_t \in N_G(Q) \cap \rho^{i-1}(Q_t)$ 
for some nontrivial subgroup 
$Q_t$ of $S$,  for each $t = 1, \ldots, m$.
For each $t$ we consider the restriction 
of $M$ to $N_G(Q) \cap N_G(Q_t)$.
Using the notation above, this yields
\[
M{\downarrow_{N_G(Q)\cap N_G(Q_t)}^G} \quad \cong \quad
L_Q \downarrow_{N_G(Q)\cap N_G(Q_t)}^{N_G(Q)} \oplus \proj
\quad \cong \quad 
L_{Q_t} \downarrow_{N_G(Q)\cap N_G(Q_t)}^{N_G(Q_t)} \oplus \proj \,.
\]
The intersection $N_G(Q) \cap N_G(Q_t)$ contains 
the center of $S$ which is a nontrivial $p$-subgroup.
Thus any indecomposable  projective module 
over $N_G(Q) \cap N_G(Q_t)$ has dimension divisible
by~$p$. Therefore, by the Krull-Schmidt Theorem,
$$
L_Q \downarrow_{N_G(Q)\cap N_G(Q_t)}^{N_G(Q)} \quad \cong \quad
L_{Q_t} \downarrow_{N_G(Q)\cap N_G(Q_t)}^{N_G(Q_t)} \,.
$$
For all $t$, it follows by induction that 
$\chi_Q(x_t) = \chi_{Q_t}(x_t) = 1$, because $x_t\in \rho^{i-1}(Q_t)$.
Thus $\chi_Q(x) = 1$, as asserted. 
\end{proof}

The main theorem of this section is a special case of the above.

\begin{thm} \label{thm:vanishrho}
Suppose that $M$ is an endotrivial $kG$-module with trivial Sylow 
restriction, {\it i.~e.} $M{\downarrow_S^G} \cong k \oplus \proj$.
Then $M{\downarrow_{\rho^{\infty}(S)}^G} \cong k \oplus \proj$.
\end{thm}

There are several immediate corollaries. 

\begin{cor} \label{cor:rhoeqnorm}
Suppose that $\rho^{\infty}(S) = N_G(S)$. Then the only  indecomposable 
$kG$-module with trivial Sylow restriction is the trivial module.
In other words, $K(G)=\{0\}$.
\end{cor}

\begin{proof}
If $M$ is a $kG$-module with trivial 
Sylow restriction, then our assumption implies that
$M{\downarrow_{N_G(S)}^G}  \cong k \oplus \proj$. 
Because $M$ is indecomposable,
it is the Green correspondent of the trivial 
module $k_{N_G(S)}$. Thus 
$M$ is the trivial module.
\end{proof}

\begin{cor} \label{cor:normq}
Suppose that $Q$ is a normal subgroup of $S$ and that for some $i$,
$N_G(S) \subseteq \rho^i(Q)$. Then the only  indecomposable
$kG$-module with trivial Sylow restriction is the trivial module.
In other words, $K(G)=\{0\}$.
\end{cor}

\begin{proof}
Under the hypothesis we have that $N_G(S) \subseteq \rho^{i+1}(S)$
and we are done by the previous corollary. 
\end{proof}

Note in the above corollary, that $Q$ being normal in $S$ does not 
assure that $N_G(S)$ is a subgroup of $N_G(Q)$. However, if $Q$ is 
characteristic in $S$ then this is a certainty. 

\begin{cor} \label{cor:charq}
Suppose that $Q$ is a characteristic subgroup of $S$ and that 
for some $i$, $N_G(Q) = \rho^i(Q)$. Then the only  indecomposable
$kG$-module with trivial Sylow restriction is the trivial module.
In other words, $K(G)=\{0\}$.
\end{cor}


\section{Abelian Sylow subgroup}\label{sec:mainthm}

\noindent
The purpose of this section is to prove 
our main theorem, which describes $K(G)$
when a Sylow $p$-subgroup $S$ of~$G$ is abelian.
Specifically, we show the following.

\begin{thm} \label{thm:main}
Suppose that a Sylow $p$-subgroup $S$ of~$G$ is abelian. Let $N = N_G(S)$.
\begin{itemize}
\item[(a)] The image of the restriction map 
$\res_{N}^G : A(G) \to A(N)$
consists exactly of all group homomorphisms 
$N_G(S)\to k^\times$ having $\rho^2(S)$ in their kernel.
\item[(b)] $K(G)\cong A(G) \cong \big(N_G(S)/\rho^2(S) \big)^*$.
\item[(c)] $\rho^2(S)=\rho^\infty(S)$.
\end{itemize}
\end{thm}

In other words, the theorem says that if the Sylow subgroup $S$ of $G$
is abelian, then the subgroup $J$ of Proposition \ref{prop:charsub}
is equal to $\rho^2(S)$ and that $K(G) \cong (N/\rho^2(S))^*$.

Note that by Proposition \ref{prop:rho-weak}, the restriction 
to $N_G(Q)$ of any weak homomorphism $\phi:G\to k^\times$
is a homomorphism having $\rho^2(Q)$ in its kernel.
Our task then is to prove a very strong converse, namely, that
any group homomorphism $\chi:N = N_G(S) \to k^\times$ 
having $\rho^2(S)$ in its kernel is the restriction to $N$
of a weak homomorphism.

We need a couple of preliminary results before beginning the 
proof. The first observation is essential to our efforts. 

\begin{lemma} \label{lem:Frattini}
Let $S$ be a Sylow $p$-subgroup of a finite group~$H$.
\begin{itemize}
\item[(a)] The inclusion map $N_H(S) \to H$ induces an isomorphism
\[
N_H(S)/ N_H(S)\cap \dmd(H) \cong H/\dmd(H) \,.
\]
\item[(b)] For any group homomorphism
$\phi: N_H(S) \to k^\times$ having $N_H(S)\cap \dmd(H)$
in its kernel, there exists a unique group homomorphism 
$\psi: H\to k^\times$ whose restriction 
to~$N_H(S)$ is equal to~$\phi$. 
\end{itemize}
\end{lemma}

\begin{proof}
By the definition of~$\dmd(H)$, the Sylow $p$-subgroup $S$ 
is contained in~$\dmd(H)$.
The Frattini argument yields $N_H(S) \dmd(H) = H$, 
proving (a). Then (b) follows from~(a).
\end{proof}

\begin{lemma} \label{lem:Burnside} (Burnside's Fusion Theorem)
Let $S$ be an abelian Sylow $p$-subgroup of~$G$ 
and suppose that $R$ is a nontrivial subgroup of~$S\cap \ls gS$.
\begin{itemize}
\item[(a)] There exists $c\in C_G(R)$ and $n\in N_G(S)$ such that $g=cn$.
\item[(b)] If $g=c'n'$ with $c'\in C_G(R)$ and $n'\in N_G(S)$, 
then there exists $d\in N_G(S)\cap C_G(R)$
such that $c'=cd$ and $n'=d^{-1}n$.
\end{itemize}
\end{lemma}

\begin{proof}
Statement (a) is essentially Burnside's theorem. 
For the proof, observe that $S$ and $\ls gS$
are Sylow $p$-subgroups of~$C_G(R)$, hence 
conjugate by an element $c\in C_G(R)$.
Then $n=c^{-1}g$ normalizes~$S$. Statement (b) 
follows by defining $d=c^{-1}c'=nn'^{-1}$.
\end{proof}

The essence of the proof of Theorem~\ref{thm:main} 
is the following. 

\begin{prop} \label{prop:essence}
Suppose that a Sylow $p$-subgroup $S$ of $G$ is abelian.
Let $N = N_G(S)$. Let $\chi:N \to k^\times$ be a homomorphism
whose kernel contains~$\rho^2(S)$. Then there is a unique 
weak homomorphism $\theta: G \to k^\times$, whose restriction
to $N$ is equal to~$\chi$.
\end{prop}

\begin{proof}
Let $Q$ be a nontrivial subgroup of~$S$, and 
let $H=C_G(Q)$, the centralizer of~$Q$ in~$G$.
Since $S$ is abelian, $S\subseteq H\subseteq N_G(Q)$.
It is clear that $\dmd(H) = [H,H]S \subseteq \dmd(N_G(Q))$,
hence $\chi$ vanishes on $N_H(S) \cap     
\dmd(H) \subseteq \rho^2(S)$ by assumption.
By Lemma~\ref{lem:Frattini}, there exists a unique group homomorphism
\[
\psi_Q: H = C_G(Q) \quad \longrightarrow \quad k^\times
\]
which coincides with $\chi$ on the subgroup 
$N_H(S)= N_G(S) \cap C_G(Q) = N \cap H$.

We define $\theta:G\to k^\times$ by the following rule. 
First, we set $\theta(g)=1$ if $S\cap \ls gS=1$.
If $S\cap \ls gS\neq \{1\}$, we use 
Lemma~\ref{lem:Burnside} and write $g=cn$,
with $c\in C_G(S\cap \ls gS)$ and $n\in N_G(S)$.  Then let
\[
\theta(g)=\psi_{S\cap \ls gS}(c) \, \chi(n) \,.
\]

In order to prove that $\theta$ is well defined, 
we must consider another decomposition $g=c'n'$,
with $c'\in C_G(S\cap \ls gS)$ and $n'\in N_G(S)$, 
and show that the algorithm produces 
the same result for $\theta(g)$.
In fact, we prove more: that the algorithm 
produces the same result even if we replace $S\cap \ls gS$ 
by a proper nontrivial subgroup.  That is, assume that $R$ 
is any nontrivial subgroup of $S\cap \ls gS$ and write
$g=c'n'$, with $c'\in C_G(R)$ and $n'\in N_G(S)$. 
We claim that
\[
\psi_{S\cap \ls gS}(c) \, \chi(n) = \psi_R(c') \, \chi(n') \,.
\]

In order to prove the claim, we observe that 
$C_G(S\cap \ls gS) \subseteq C_G(R)$ 
because $R\subseteq S\cap \ls gS$.
So we have two decompositions $g=cn=c'n'$ 
with $c,c'\in C_G(R)$ and $n,n'\in N_G(S)$.
By Lemma~\ref{lem:Burnside}, there exists 
$d\in N_G(S)\cap C_G(R)$ such that $c'=cd$ and $n'=d^{-1}n$.
Thus,
\[
\psi_R(c') \, \chi(n') = \psi_R(cd) \, \chi(d^{-1}n) = 
\psi_R(c) \, \psi_R(d) \, \chi(d)^{-1} \, \chi(n) = 
\psi_R(c) \, \chi(n) \,,
\]
because $d\in N_G(S)\cap C_G(R)$ and we know 
that $\psi_R$ and $\chi$ coincide on
$N_G(S)\cap C_G(R)$. We next observe that the uniqueness of 
$\psi_{S\cap \ls gS}: C_G(S\cap \ls gS) \to k^\times$
implies that it must be equal to the 
restriction to $C_G(S\cap \ls gS)$
of the group homomorphism $\psi_R: C_G(R)\to k^\times$.
Therefore $\psi_R(c) = \psi_{S\cap \ls gS}(c)$, 
completing the proof of the claim.

Our next task is to prove that $\theta$ is a weak homomorphism. 
If $s\in S$, then $S\cap \ls sS=S$ and we have
$C_G(S)\subseteq N_G(S)$. We use the decomposition $s=1\cdot s$ 
with $1\in C_G(S)$ and $s\in N_G(S)$
and we get $\theta(s)=\chi(s)$. But $\chi$ vanishes on~$S$ because 
$S\subseteq \dmd(N)$.
Therefore $\theta(s)=1$, proving the first 
condition for a weak homomorphism.
The second condition is obvious since 
$\theta(g)=1$ if $S\cap \ls gS=1$, by definition.

Assume that $a,b\in G$ satisfy $S\cap \ls aS \cap \ls {ab}S\neq \{1\}$. 
We must show that $\theta(ab)=\theta(a)\theta(b)$. 
Let $R=S\cap \ls aS \cap \ls {ab}S$, and 
note that $R\subseteq S\cap \ls aS$.
Using what we have proved, we can write
\[
\theta(a)= \psi_R(c)\,\chi(n) \,,\quad \text{where} 
\quad  a=cn,\, c\in C_G(R), \, n\in N_G(S) \,.
\]
Next notice that $R\subseteq \ls aS\cap \ls {ab}S$, 
so that $R^a\subseteq S\cap \ls bS$. We write
\[
\theta(b) =  \psi_{R^a}(d)\,\chi(m) \,, \quad 
\text{where} \quad b=dm, \,d\in C_G(R^a), \, m\in N_G(S) \,.
\]
It follows that $\ls ab=\ls ad \, \ls am$, with 
$\ls ad\in C_G(R)$, $\ls am\in N_G(\ls aS)$. Therefore
\[
ab=\ls ab \,a= \ls ad \, \ls am \, a = 
\ls ad \,c \,c^{-1} \,a \,m =  \ls ad \,c \,n\,m \,,
\]
because $c^{-1}a=n$.
But $\ls ad \,c \in C_G(R)$, because $d\in C_G(R^a)$, 
and $nm\in N_G(S)$. So we obtain
\[
\theta(ab)= \psi_R(\ls adc) \, 
\chi(nm) = \psi_R(\ls ad)\,\psi_R(c)\,\chi(n)\,\chi(m)
= \psi_R(c)\,\chi(n)\,\psi_R(\ls ad)\,\chi(m) \,.
\]

Now we claim that $\psi_R(\ls ad)=\psi_{R^a}(d)$. 
Writing $a=cn$, we first find that
\[
\psi_R(\ls ad)= \psi_R(c\,\ls nd\,c^{-1})=
\psi_R(c)\,\psi_R(\ls nd) \, \psi_R(c^{-1}) = 
\psi_R(\ls nd) \,,
\]
because $k^\times$ is commutative. Moreover, writing 
$\conj_n(x)=nxn^{-1}$, we observe that the composite
\[
\xymatrix{
{C_G(R^a) = C_G(R^n)} \ar[r]^{\qquad\conj_n}  
& C_G(R) \ar[r]^{\;\psi_R} &  k^\times
}
\]
has a restriction to $N_G(S)\cap C_G(R^n)$ equal to
\[
\xymatrix{
N_G(S)\cap C_G(R^n) \ar[r]^{\;\conj_n}  & 
N_G(S)\cap C_G(R) \ar[r]^{\qquad\chi} &  k^\times.
}
\]
Hence, it is equal to the map
\[
\xymatrix{
N_G(S)\cap C_G(R^n) \ar[r]^{\qquad\chi} &  k^\times}
\]
because $n\in N_G(S)$ and 
$\chi(nxn^{-1})=\chi(n)\chi(x)\chi(n^{-1})=\chi(x)$ 
by commutativity of~$k^\times$.
It follows that 
$\psi_R\circ\conj_n : C_G(R^n)\to k^\times$ is the 
unique extension of $\chi : N_G(S)\cap C_G(R^n) \to k^\times$,
hence equal to the homomorphism 
$\psi_{R^n} : C_G(R^n)\to k^\times$.
In other words, 
$\psi_R(\ls nd)=\psi_R(\conj_n(d)) =\psi_{R^n}(d)$.
Finally
\[
\psi_R(\ls ad)=\psi_R(\ls nd)=\psi_{R^n}(d)=\psi_{R^a}(d) \,,
\]
as claimed.

Returning to the computation of $\theta(ab)$, we find
\[
\theta(ab)= \psi_R(c)\,\chi(n)\,\psi_R(\ls ad)\,\chi(m)
= \psi_R(c)\,\chi(n)\,\psi_{R^a}(d)\,\chi(m)
=\theta(a) \,\theta(b) \,.
\]
This completes the proof of the third 
condition for a weak homomorphism.
Thus $\theta:G\to k^\times$ is a weak homomorphism,
and we have proved the proposition. 
\end{proof}

\begin{proof}[Proof of Theorem \ref{thm:main}]
To prove (a), we recall that any weak homomorphism 
on $N$ is a homomorphism. The image under the 
restriction map $\res_{N}^G : A(G) \to A(N)$ of any
element of $A(G)$ must be a homomorphism with 
$\rho^2(S)$ in its kernel. The previous proposition
says that the restriction map must be injective and
surjective onto this subset. 

Statement (b) follows immediately from (a) and the 
injectivity of the restriction map 
(Proposition~\ref{prop:injective}).
The proof of (c), that is
the equality $\rho^2(S)=\rho^\infty(S)$,
follows from~(a) and the fact that
$\rho^\infty(S)$ is in the kernel of any weak 
homomorphism, by Theorem~\ref{thm:vanishrho}.
\end{proof}

All of the experimental evidence suggests that something
like Theorem \ref{thm:main} should be true in general. 
Hence, we suggest the following question.

\begin{conjj}
Suppose that $G$ is any finite group with Sylow
$p$-subgroup~$S$. Let $N = N_G(S)$, and let $J \subseteq N$ 
be the intersection of the kernels of all one-dimensional 
$kN$-modules $U$ such that the Green correspondent
of $U$ is an endotrivial $kG$-module. Is 
$J = \rho^\infty(S)$?
\end{conjj} 

While it might be possible to prove such a conjecture 
by means similar to those in the proof of 
Theorem \ref{thm:main}, we should point out that at least
two difficulties arise. The first is that the fusion
theorem of Burnside does not hold in greater generality. 
It would have to be replaced with something like Alperin's 
fusion theorem, whose conditions are more complicated. 
In addition, as we see in Section \ref{sec:fusion}, even
a stringent assumption such as control of fusion by the
normalizer of the Sylow subgroup $S$ does not readily 
lead to a generalization. Other assumptions seem to be 
necessary for the proof. 

A second difficulty to generalizing Theorem \ref{thm:main}
is that the equality $\rho^2(S)=\rho^\infty(S)$ does 
not hold in general. An example is the group $G_2(5)$
for $k$ a field of characteristic~3. Computer calculations 
using MAGMA~\cite{BC} show
that $\rho^3(S) = \rho^\infty(S) = N_G(S)$, but 
$\rho^2(S)$ is a proper subgroup of~$N_G(S)$.


\section{The cyclic case}  \label{sec:cyclic}

\noindent
If a Sylow $p$-subgroup $S$ of $G$ is cyclic, then the structure 
of $K(G)$ is determined in Theorem~3.6 of~\cite{MT}.
In this section, we show that this result can be recovered 
using Theorem~\ref{thm:main}. We prove the following. 

\begin{thm} \label{thm:cyclic}
Suppose that a Sylow $p$-subgroup $S$ of~$G$ is cyclic.
Let $Z$ be the unique subgroup of~$S$ of order~$p$.
Then $K(G)\cong K(N_G(Z)) \cong     
\big(N_G(Z)/\dmd(N_G(Z)) \big)^*$.
\end{thm}

\begin{proof}
For any subgroup $Q$ such that $Z\subseteq Q\subseteq S$, 
we have that $N_G(Q)\subseteq N_G(Z)$.
Hence $\dmd(N_G(Q)) \subseteq \dmd(N_G(Z))$, and
\[
\rho^2(S) = N_G(S)\cap \dmd(N_G(Z)) \,.
\]
By Lemma~\ref{lem:Frattini} applied to the 
subgroup $H=N_G(Z)$, we have an isomorphism
\[
N_G(S)/\rho^2(S) = N_G(S)/N_G(S)\cap \dmd(N_G(Z))  
\cong N_G(Z)/\dmd(N_G(Z)) \,.
\]
Since $K(G)\cong \big(N_G(Z)/\rho^2(S)\big)^*$ 
by Theorem~\ref{thm:main}, we obtain
\[
K(G)\cong \big(N_G(Z)/\dmd(N_G(Z)) \big)^* \,,
\]
as required. This is also isomorphic to $K(N_G(Z))$ 
by Proposition~\ref{prop:normal}.
\end{proof}

The isomorphism $K(G) \cong K(N_G(Z))$ follows actually directly from the isomorphism $T(G) \cong T(N_G(Z))$,
which is a consequence of the fact that $N_G(Z)$ is strongly $p$-embedded in~$G$ (see Lemma~3.5 in~\cite{MT}).


\section{Control of fusion} \label{sec:fusion}

\noindent
Assume that the normalizer $N_G(S)$ of a 
Sylow $p$-subgroup $S$ of $G$ controls $p$-fusion.
One may wonder if Theorem~\ref{thm:main} 
still holds under this assumption.
The analysis of the proof shows that we need more, as follows.

\begin{thm} \label{thm:fusion}
Suppose that the normalizer $N = N_G(S)$ 
of a Sylow $p$-subgroup $S$ of $G$ controls $p$-fusion.
Assume, in addition, that for any subgroup $Q$ of~$S$, we have 
$N_H(S)\dmd(H) =H$
where $H=C_G(Q)$ and $N_H(S)=N_G(S)\cap H$. 
Then the following hold~:
\begin{itemize}
\item[(a)] The image of the restriction map 
$\res_{N_G(S)}^G : A(G) \to A(N_G(S))$
consists exactly of all group homomorphisms 
$N_G(S)\to k^\times$ having $\rho^2(S)$ in their kernel.
\item[(b)] $K(G)\cong A(G) \cong \big(N_G(S)/\rho^2(S) \big)^*$.
\item[(c)] $\rho^2(S)=\rho^\infty(S)$.
\end{itemize}
\end{thm}

\begin{proof}
The proof is exactly the same as that of 
Theorem~\ref{thm:main}, with the following observations.
The assumption on each group $H=C_G(Q)$ 
implies that the conclusions of Lemma~\ref{lem:Frattini} hold.
Thus the use of Lemma~\ref{lem:Frattini} remains valid.
More precisely, for any group homomorphism 
$\chi:N_G(S)\to k^\times$ vanishing  
on~$N_H(S)\cap \dmd(H)$,
there exists a unique group homomorphism 
$\psi_Q:C_G(S)\to k^\times$ which coincides 
with $\chi$ on the subgroup
$N_H(S)=N_G(S)\cap H$.
Here $S$ is not necessarily contained in~$H$ 
(while it is when $S$ is abelian),
so the Frattini argument cannot be applied as it was in Lemma~\ref{lem:Frattini}.
However, our assumption allows us to make the argument work.

On the other hand, the assumption on 
control of fusion means exactly
that the conclusions of Lemma~\ref{lem:Burnside} hold.
Thus, the use of Lemma~\ref{lem:Burnside} remains valid, and
the whole proof goes through.
\end{proof}


\section{Examples}  \label{sec:examples}

\noindent
If $H$ is a strongly $p$-embedded subgroup 
of~$G$, then any one-dimensional representation of~$H$
has a Green correspondent which is endotrivial 
(see Proposition~2.8 and Remark~2.9 in~\cite{CMN1}).
This fact was used to produce torsion endotrivial 
modules of dimension greater than one in various cases,
in particular for groups of Lie type of rank one in 
the defining characteristic (Proposition~5.2 in~\cite{CMN1})
and for groups with a cyclic Sylow $p$-subgroup 
(Lemma~3.5 in~\cite{MT}).

However, there are other cases when torsion endotrivial 
modules of dimension greater than one occur.
Several examples are given in~\cite{LM2} for various sporadic groups.
The purpose of this section is to provide two 
explicit examples for classical groups
with an abelian Sylow $p$-subgroup. 
We first start by an easy case.

\begin{exm} {\bf $PSL(2,q)$ in characteristic 2.}
Let $G=PSL(2,q)$ in characteristic~2 and assume that $q \equiv 3$ or 5 modulo~8,
so that a Sylow 2-subgroup $S$ of~$G$ is a Klein four-group.
Then $C_G(S)=S$ has index~3 in~$N_G(S)$, hence $\dmd(N_G(S)) =S$.
Any subgroup~$C$ of order~2 satisfies $N_G(C)=C_G(C)$ and
\[
S\subseteq N_G(S)\cap \dmd(N_G( C)) \subseteq 
N_G(S)\cap N_G(C) = N_G(S)\cap C_G(C) = S \,.
\]
Therefore $N_G(S)\cap \dmd(N_G(C)) =S$
and it follows that $\rho^2(S)=\rho^1(S)=S$.
By Theorem~\ref{thm:main}, $K(G)$ is the dual 
group of~$N_G(S)/S$, which is cyclic of order~3.
Hence $TT(G)=K(G)\cong \bZ/3\bZ$.
\end{exm}

For our second example, we compute the torsion part of the group
of endotrivial modules over the group $G = \PSL(3,q)$ in characteristic~3,
in the case that $q \equiv 4$ or 7 modulo~9. In this case, the 
Sylow 3-subgroup of $G$ is elementary abelian of order~9.
The point is to show that the torsion subgroup $TT(G)=K(G)$
of the group of endotrivial module is 
isomorphic to $(\bZ/2\bZ)^2$, which has three nontrivial elements. 

For notation, we use an overline to indicate the class in $G$ of an 
element in $H = \SL(3,q)$. We fix the following elements of $H$.
For notation, let $\zeta$ denote a cubed root of unity in $\bfq$. 
\[
a = \begin{bmatrix} 1 & 0 & 0  \\
		0 & \zeta & 0 \\
		0 & 0 & \zeta^2  \end{bmatrix}, \quad 
x = \begin{bmatrix} 0 & 1 & 0  \\
                0 & 0 & 1 \\
                1 & 0 & 0  \end{bmatrix}, \quad 
u = \begin{bmatrix} 1 & 1 & 1  \\
                1 & \zeta & \zeta^2 \\
                1 & \zeta^2 & \zeta  \end{bmatrix}, \quad 
v = \begin{bmatrix} \zeta & 1 & 1  \\
                \zeta & \zeta & \zeta^2 \\
                1 & \zeta & 1  \end{bmatrix}.
\]
Then it is not difficult to verify the following. 

\begin{lemma}\label{lem:generators1}
The subgroup $\langle a,x \rangle$ is a Sylow 3-subgroup of 
$H = \SL(3,q)$ and hence the Sylow 3-subgroup of $G$ is 
$S = \langle \ovl{a}, \ovl{x} \rangle$.  In addition, we have the relations~:
\begin{enumerate}
\item $u^{-1}xu = a,$
\item $u^{-1}au = x^{-1},$
\item $v^{-1}xv = a^2v,$
\item $v^{-1}av = \zeta (ax)^{-1}.$
\end{enumerate}
In particular, the elements $\ovl{u}$ and $\ovl{v}$ are in the 
normalizer of $S$, and each acts on $S$ by exchanging the four 
maximal subgroups in pairs.  The commutator $\sigma = u^{-1}v^{-1}uv$
acts on $S$ by inverting $\ovl{a}$ and~$\ovl{x}$, and hence also 
inverting every nonidentity element. 
\end{lemma}

Now we consider the centralizers and normalizers. Let $z \in H$
be the scalar matrix with nonzero entries equal to $\zeta$. Then
$\langle z \rangle$ is the kernel of the natural homomorphism
of $H$ onto $G$. A general principle here is that if $y$ is an
element of $H$ such that $\ovl{y}$ commutes with $\ovl{x}$, then for 
some $j$, $a^jy$ is in the centralizer of $x$, and the same holds 
with $a$ and $x$ exchanged. 

It is easy to see that the centralizer of $a$ in $H$ is the Levi
subgroup $L$ of all diagonal matrices of determinant one. Then the 
normalizer of $\langle a, z \rangle$ is the normalizer of $L$ which is
generated by $L$, $x$ and the element $\sigma$, the commutator 
of $u$ and $v$.
The centralizer of $x$, which has the same order as that of $a$,
consists of all elements of the form 
\[ 
\begin{bmatrix} c & d & e \\ e & c & d \\ d & e & c \end{bmatrix}
\]
having determinant one ($c^3+d^3+e^3 - 3cde =1$). Then, the 
normalizer of $\langle x, z \rangle$ is generated by this centralizer,
$a$ and $\sigma$. 

Thus we can prove the following. 

\begin{prop}\label{prop:normalizer}
The Sylow 3-subgroup $S$ is self centralizing. Its normalizer 
is generated by $\ovl{a}, \ovl{x}, \ovl{u}, \ovl{v}$ and~$\ovl{\sigma}$.
\end{prop}

\begin{proof}
We can see that the centralizer of $S$ is generated by the 
classes $\ovl{a}, \ovl{x}$ and the classes of the intersection 
of the centralizers of $a$ and $x$ in~$H$. However, this 
intersection consists only of the elements of $\langle z \rangle$.
The elements $\ovl{u}$ and $\ovl{v}$ permute the maximal subgroups
of $S$ transitively. So suppose that $\ovl{y}$ is an element of 
the normalizer of $S$ and $y$ a premiage of $\ovl{y}$ in $H$ . 
By replacing  $y$ by its multiple with 
a power of $u$ and/or a power of $v$ we may assume that that $\ovl{y}$ 
normalizes $\langle \ovl{a} \rangle$. By replacing $\ovl{y}$
with its multiple with $\sigma$, if necessary, we may assume that 
$\ovl{y}$ centralizes~$\ovl{a}$. Multiplying $y$ by a power of $x$
if necessary, we may assume that $y$ has the form 
\[ 
\begin{bmatrix} r & 0 & 0 \\  0 & s & 0 \\   0 & 0 & t \end{bmatrix}
\]
where $rst = 1$. Thus we have that 
\[
yxy^{-1} \quad = \quad \begin{bmatrix} 0 & r^2t & 0 \\
 					0 & 0 & rs^2 \\
					st^2 & 0 & 0 \end{bmatrix}
\]
The point of this is that $\ovl{y}\ovl{x}\ovl{y}^{-1}$  must be 
one of the elements $\ovl{x}, \ovl{a}\ovl{x}$ or 
$\ovl{a}^2\ovl{x}$ and {\em cannot} be
$\ovl{x}^2, \ovl{a}\ovl{x}^2$ or~$\ovl{a}^2\ovl{x}^2$. Because 
$y$ centralizes $a$, conjugation by $\ovl{y}$ is an automorphism of order
either one or three on $S$. Order 3 is not possible. Consequently,
$\ovl{y}\ovl{x}\ovl{y}^{-1} = \ovl{x}$ and $\ovl{y}$ centralizes $S$.
This proves the proposition. 
\end{proof}

Now we are ready for the main theorem. 

\begin{thm} \label{thm:psl3}
Assume that $G = \PSL(3,q)$ where $q$ is congruent to $4$ or $7$
modulo~$9$.
\begin{itemize}
\item[(a)] Let $S$ be a Sylow 3-subgroup of $G$. Then 
$\rho^\infty(S)=\rho^1(S) = [N_G(S), N_G(S)]$.
\item[(b)] $TT(G)=K(G) \cong (\bZ/2\bZ)^2$.
\end{itemize}
\end{thm}

\begin{proof}
(a) We note first that 
$\rho^1(S) = \dmd(N_G(S)) = [N_G(S), N_G(S)]$ by
definition and the fact that $[N_G(S), N_G(S)]$ has index prime to~3 in~$N_G(S)$
(namely index~4).
We also have that for any subgroup $U$ of order~$3$ in $S$,
$N_G(U) \cap N_G(S) \subseteq [N_G(S), N_G(S)]$.
Consequently, again from the definition, we have that 
$\rho^n(S) = [N_G(S), N_G(S)]$ for all $n$.

(b) Since a Sylow 3-subgroup is abelian of rank two, $TT(G)=K(G)$
is isomorphic to the dual group of 
$N_G(S)/ \rho^2(S)$, by Theorem~\ref{thm:main}.
By part~(a), $\rho^2(S)=[N_G(S), N_G(S)]$
and by Proposition~\ref{prop:normalizer}, 
$N_G(S)/[N_G(S), N_G(S)] \cong (C_2)^2$,
a Klein four-group, generated by the classes of $\ovl{u}$ and~$\ovl{v}$.
Its dual group (in additive notation) is isomorphic to $(\bZ/2\bZ)^2$.
\end{proof}  

\begin{rem}
In the case that $q=4$, the normalizer of the Sylow subgroup $S$ of 
$G$ is strongly 3-embedded, and the result of the theorem could 
be deduced from that fact. In all other cases, $N_G(S)$ is not strongly
3-embedded and it is not strongly 3-embedded in any subgroup of 
$G$ that properly contains it. 
\end{rem}

Finally, it is not difficult to perform the computations of the 
subgroups $\rho^i(Q)$ for all subgroups $Q$ of the Sylow subgroup
$S$ of $G$ on a computer using a standard computer algebra system.
From this computation, the structure of $K(G) = TT(G)$ can in many cases 
be deduced using Theorem~\ref{thm:main} or something similar. Below
are a few calculations using MAGMA~\cite{BC}. In most cases, 
only a few seconds of computing time was required. The computing 
time depends on such things as the size of the permutation 
representation of $G$ and the number of subgroups of $S$. Here 
we list only groups where $K(G)$ is not trivial. The results 
should be compared with those of \cite{LMS, LM1, LM2}. The notation
for the groups is the Atlas notation.  

\begin{align*}
&TT(G)  & Group \qquad \qquad \qquad \qquad p \\
&===== & ===== \qquad \qquad \qquad \qquad =\\
&\bZ/2\bZ &  M_{23}, Ru  \quad  \text{  in characteristic 3}\\
& &  J_2, Suz, 2Suz, 6Suz, Fi_{22}, Fi_{23} \quad  \text{  in characteristic 5}\\
&\bZ/2\bZ \times \bZ/2\bZ &
M_{11}, M_{22}, M_{23}, HS \quad  \text{  in characteristic 3}\\
&\bZ/4\bZ & 2Ru \quad  \text{  in characteristic 3}\\
& & Co_3, Sz_{32} \quad  \text{  in characteristic 5}\\
&\bZ/2\bZ \times \bZ/4\bZ & 2M_{22}, 4M_{22} \quad  \text{  in characteristic 3}\\
&\bZ/8\bZ & McL \quad  \text{  in characteristic 5}\\
&\bZ/24\bZ & 3McL \quad  \text{  in characteristic 5}\\
\end{align*}

\bigbreak
\noindent
Jon F. Carlson
\par\nobreak\noindent
Department of Mathematics,
University of Georgia,
Athens, GA 30602,
USA.
\par\nobreak\noindent {\tt jfc@math.uga.edu}

\medskip
\noindent
Jacques Th\'evenaz
\par\nobreak\noindent
Section de math\'ematiques,
EPFL,
Station 8,
CH-1015 Lausanne,
Switzerland.
\par\nobreak\noindent {\tt Jacques.Thevenaz@epfl.ch}

\end{document}